\documentclass[12pt]{amsart}
\usepackage[T1]{fontenc}
\usepackage[utf8]{inputenc}
\usepackage{geometry}                
\usepackage{graphicx}
\usepackage{amssymb}
\usepackage{xspace}
\usepackage{amsmath}
\usepackage{amsthm}
\usepackage{bm}
\usepackage{tabularx}
\usepackage{epstopdf}
\usepackage{xcolor}
\definecolor{deepskyblue}{rgb}{0.0, 0.75, 1.0}
\usepackage{color}
\usepackage{amsmath,amsthm,amscd,amssymb, enumerate}
\usepackage{latexsym}
\usepackage[colorlinks,citecolor=red,pagebackref,hypertexnames=false]{hyperref}

\theoremstyle{remark}
\newtheorem{assumption}{Assumption}
\newtheorem{remark}{Remark}
\newtheorem{lemma}{Lemma}
\newtheorem{theorem}{Theorem}
\newtheorem{definition}{Definition}

\newcommand{\bgamma}{\boldsymbol{{\gamma}}}
\newcommand{\bxi}{\boldsymbol{\xi}}

\makeatletter
\newcommand{\tpitchfork}{%
  \vbox{
    \baselineskip\z@skip
    \lineskip-.52ex
    \lineskiplimit\maxdimen
    \m@th
    \ialign{##\crcr\hidewidth\smash{$-$}\hidewidth\crcr$\pitchfork$\crcr}
  }%
}
\makeatother

\def\R{\mathbb R}

\def\be{\begin{equation}}
\def\ee{\end{equation}}
\def\beqa{\begin{eqnarray}}
\def\eeqa{\end{eqnarray}}

\def\coi{c_0^{-1}}
\def\cot{c_0^{-2}}
\def\nx{d_{\bm x}}
\def\dR{\dot{R}}
\def\ddR{\ddot{R}}

\def\ie{{\it i.e.}}
\def\eg{{\it e.g.}}

\def\bm{\mathbf }

\title[Microlocal analysis of Doppler SAR]{Microlocal analysis of \\ Doppler Synthetic Aperture Radar}
\author[Raluca Felea, Romina Gaburro, Allan Greenleaf and  Clifford Nolan]{}


\email{rxfsma@rit.edu}
\email{Romina.Gaburro@ul.ie}
\email{allan@math.rochester.edu}
\email{Clifford.Nolan@ul.ie}

\begin{document}
\maketitle

\centerline{\scshape Raluca Felea}
\medskip
{\footnotesize
 \centerline{School of Mathematical Sciences}
   \centerline{Rochester Institute of Technology}
   \centerline{ Rochester, NY, 14623, USA}
} 

\medskip

\centerline{\scshape Romina Gaburro}
\medskip
{\footnotesize
 \centerline{ Department of Mathematics and Statistics}
  \centerline{ and Health Research Institute}
   \centerline{University of Limerick}
   \centerline{Limerick, V94 T9PX, Ireland}
}

\medskip

\centerline{\scshape Allan Greenleaf}
\medskip
{\footnotesize
 \centerline{Department of Mathematics}
   \centerline{University of Rochester}
   \centerline{Rochester, NY, 14627, USA}
}

\medskip

\centerline{\scshape Clifford Nolan}
\medskip
{\footnotesize
 \centerline{ Department of Mathematics and Statistics}
  \centerline{ and Health Research Institute}
   \centerline{University of Limerick}
   \centerline{Limerick, V94 T9PX, Ireland}
}

\bigskip

\begin{abstract}
We study the existence and suppression  of artifacts for a Doppler-based Synthetic Aperture Radar (DSAR) system.
The idealized air- or space-borne system transmits a continuous wave at a fixed frequency
and a co-located receiver measures the resulting scattered waves;
a windowed Fourier transform  then converts the raw data into a function of two variables: slow time and frequency.
Under simplifying assumptions, we analyze the linearized forward scattering map
and the feasibility of inverting it via filtered backprojection,
using techniques of microlocal analysis which robustly describe how sharp features in the target appear in the data.
For DSAR with a  straight flight path, there is, as with conventional SAR, a left-right ambiguity artifact in the DSAR image,
which can be avoided via beam forming to the left or right.
For  a circular flight path, the artifact has a more complicated structure,
but filtering out echoes coming from straight ahead or behind the transceiver,
as well as those outside a critical range, produces an artifact-free  image.
 We show that these results  are qualitatively robust; although initially derived under an approximation
widely used  for range-based SAR,
they are either structurally stable or robust with respect to  a more accurate model.

\end{abstract}

\maketitle



\section{Introduction}
Synthetic aperture radar (SAR) systems transmit electromagnetic waves from an airborne or spaceborne antenna,
which then  scatter from objects of interest on the terrain,
e.g.,  roads, vehicles,  buildings or natural features.
The scattered waves are  measured by a receiver either co-located with the  transmitting antenna (monostatic SAR)
or on one or more other platforms (bi- or multi-static).
Radar transmitters  operate in a variety of modes, with the emitted waves  ranging from wideband pulses
to ultra-narrowband, continuous wave (CW) signals.
The Doppler Synthetic Aperture Radar (DSAR) system we analyze  in this paper is of the latter type,
 with narrow temporal windowing superimposed  on the received scattered waves
generated by a single frequency transmitted wave.
(This approach to applying temporal windows to CW signals has some similarity to what is known as
pseudo-pulse processing in the radar literature, which has been used previously for SAR \cite{Rigling06}.)
Possible uses of DSAR include low power applications and imaging through media
which are either dispersive or have frequency-dependent attenuation.

The narrowband nature of the DSAR waveforms allows accurate measurement of Doppler frequency shifts,
which in turn can be used to obtain accurate measurements of relative velocity.
Knowledge of relative velocity between stationary scattering objects and the antenna,
whose trajectory is assumed known,  is then used to locate and image the  objects.
The dual concept of  using Doppler shifts to
 image rotating or moving objects, using measurements by a stationary  antenna, was proposed some time ago
  by Thomson and Ponsonby \cite{TP}; see also \cite{RM}.
More recently, the use of a moving antenna over a horizontal terrain was proposed in Borden and Cheney \cite{BoCh04}.
In Wang and Yazici \cite{WangYaz2012}, this approach was named Doppler Synthetic Aperture Radar  and examined for
bistatic data acquisition;
related ideas were further developed in \cite{CBG, SFL, WHBBC,YLY,WangYaz2014,YSY18}.
The reader can also consult these
papers  for more extensive surveys of the earlier literature.

In contrast to DSAR, in standard SAR  the transmitted fields often consist of
 a train of pulses of short temporal duration and complicated spectrum;
the short duration allows for high-resolution estimates of the time between
transmitting and receiving antennas (whether in monostatic, bi-static or multi-static geometries),
and this travel time is used to estimate the distance (range) between the antenna and the scattering
objects.
This process is repeated from many locations along the transmitter flight path, providing distance estimates from many different
viewing angles, from which both the range and the along-track position of scattering objects can be determined.
\medskip

In this paper,  we study monostatic DSAR using methods of  microlocal analysis.
Microlocal techniques allow one to  develop  backprojection algorithms for reconstruction of features on the ground (which we call a {\it scene}) up to smooth errors,
as well as analyze obstructions to such reconstruction in the form of imaging {\it artifacts} {(fictitious features in reconstructions of a scene)}.
Such methods have been highly successful in either revealing artifacts,
arising in conventional monstatic and bistatic SAR, or in rigorously explaining artifacts that have already been observed. See \cite{F04},   \cite{NS97} and \cite{StefUhl13} for examples of artifacts that are either predicted or explained by microlocal analysis. Also see \cite{FeleaRemov} for an example of how to ameliorate the effects of artifacts.
We  describe what types of artifacts appear in DSAR and obtain sufficient conditions under which they can be excluded,
 allowing for robust imaging of sharp features  on the ground (such as walls and edges) via filtered backprojection.

For simplicity, we treat  the problem of  imaging a stationary scene located on a flat terrain.
After considering the structure of DSAR for general flight paths, we focus on two model  trajectories,
namely a straight flight path and a circular one.
For these we characterize what artifacts arise in the imaging,
and describe criteria for avoiding them (see Sections 6 and 7).

More precisely, in this paper we  first use the Born approximation and other approximations to create a forward map $\mathcal F$
taking the scene $V(\bm x)$ to be imaged to a windowed Fourier transform $W(s,\omega)$ of the  DSAR data.
(Here, $\bm x$ are coordinates on the ground, $s$ is the slow time parametrizing the transceiver's flight path  $\gamma(s)$ (assumed smooth) and $\omega$ denotes frequency.)
We then analyze $\mathcal F$,
showing  that it is a Fourier integral operator, and study   the implications
of its  microlocal geometry for imaging  by filtered backprojection.

For the linear  trajectory,  we show in Theorem \ref{thm line} that
there is a left-right artifact, similar to that which arises in monostatic SAR (see Remark \ref{rem linear}), and
which can be avoided using beam forming to the left or right of the flight track.
Furthermore, it is shown in Sec. \ref{sec beyond} that this conclusion still holds under
a more refined and realistic model,
indicating that our conclusions are robust.
For a circular flight path, the geometry is different and the artifacts are more complicated (see Remark \ref{rem circular}),
but  in Theorem \ref{thm circle} we are able to precisely characterize the artifacts, and then give criteria for avoiding  them (\eg, see
\eqref{inj_cond} in Lemma \ref{lem criterion}).

\section{The Doppler SAR model}\label{sec Doppler}

We consider an aircraft- or satellite-borne antenna,
transmitting a time-harmonic signal $e^{-i \omega_0 t}$ that scatters off the terrain and is then detected with the same antenna.
The received signal is then used to produce an image of the terrain;
 we follow \cite{BoCh04} closely for the initial modeling;
in particular we ignore polarization and use the scalar, time-domain wave equation,
\begin{equation}\label{wave equation}
\left(\nabla^2 -\frac{1}{c^2(\bm y)}\partial^{2}_t\right)E(t,\bm y)=f(t, \bm y),
\end{equation}
 where $\bm y \in \R^3$, $E$ is (one component of the) electric field,
 $f$ describes the source, and the function $c$ is the wave propagation speed.

 For the majority of  this paper, we take the source $f$ to be of constant angular frequency,  of the form $f(t, \bm y) = e^{-i \omega_0 t} \delta(\bm y - \bm \gamma(t))$, where
 ${\gamma} :=\left\{\bm\gamma(t)\:|\: t_{min}<t<t_{max}\right\}$ is
the curve describing the antenna flight path.
In particular,  in this model we assume that the antenna radiates isotropically.
More generally, following \cite{CB} and \cite{NC02}, we recall below how an explicit antenna beam pattern can be incorporated in $f$, but this will only affect the amplitudes and not the phase function and geometry.
\medskip

To avoid irrelevant degeneracies, we impose some assumptions,
\medskip

\begin{assumption} \label{height_above}
The flight  path  is strictly above the ground and,
in the region between the flight path and the ground,
 $c(\bm y)=c_0$, the constant speed of light in dry air.
\end{assumption}
\medskip

We define the \textit{reflectivity function}, $\frac{1}{c_0^2}-\frac{1}{c^2(\bm y)}$,
which encodes  changes in the propagation speed.
Electromagnetic waves attenuate rapidly as they propagate into the earth, which is modeled by the following assumption:
\medskip

\begin{assumption}\label{reflectivity}
The reflectivity function  is of the form $V(\bm y' )\delta(y_3),\, \bm y'\in\R^2$.
\end{assumption}
\medskip

We provide below a very brief sketch of the linearized model of the scattered waves and we refer the reader to \cite{NC02} for more details.  If we formally linearize (\ref{wave equation}) by writing $c=c_0+\delta c, E=E_0+\delta E$, then the {\em scattered field}, $\delta E$, approximately satisfies
\begin{equation} \label{linearized wave equation}
\left(\nabla^2 -\frac{1}{c^2(\bm y)}\partial^{2}_t\right)\delta E(t,\bm y) = - V(\bm y,t)\partial_t^2E_0(t,\bm y)
\end{equation}
where the {\em incident field}, $E_0$, satisfies (\ref{wave equation}) with $c$ replaced by $c_0$. This is the Born approximation of the scattered field. If we wish to include an antenna beam pattern in (\ref{wave equation}), we set $f(t,\bm y)=J_s(\bm y)e^{-i\omega_0t}$, where $J_s(\bm y)$ is related to the current distribution on the antenna.

Finally, for the time being, we make the following simplifying assumption:

\begin{assumption}\label{startstop}
{\bf Start-stop approximation:} The speed of wave propagation is so large, relative to the motion of the
transceiver along $\gamma$, that the point where the wave is detected by the receiver after scattering off the terrain is the
same as where the transmitter emitted it.
\end{assumption}

The start-stop approximation is common in range-based SAR.
There, the emitted wave consists of a chain of short duration pulses and, at least for airborne systems, the speed of the platform
is small relative to the speed of light; see \cite{Tsy09,CB2011} for further discussion.
In our setup, the start-stop approximation at first glance seems much harder to justify: rather than using short duration pulses
{ (typically with  complex spectra)}, it uses a long duration/narrow bandwidth signal, idealized as continuous wave (CW).
We make this assumption because the calculations needed for the microlocal techniques
used become far more complicated without it.
However, we are able to partially justify the start-stop approximation for DSAR, {\it ex post facto},
by showing that the canonical relations associated to the  linearized forward scattering map are either structurally stable
(in the case of a circular flight path) or robust to a first-order correction (in the case of a linear flight path, as shown in Sec. \ref{sec beyond}).
\smallskip

 For $\bm x\in\mathbb{R}^2$, let
\begin{equation}\label{def R}
\bm{R}(t):= (\bm x,0)-\bm\gamma(t)\quad\textnormal{and}\quad R(t):=|\bm{R}(t)|.
\end{equation}
The start-stop approximation yields the total time of travel for a  transmitted wave, $T_{tot}$,   as follows:
the downward travel time of the wave  emitted  at time $t_{tr}=t$ from $\bm\gamma(t)$   to $\bm x$
is $\coi R(t)$; the wave is thus incident to $\bm x$ at time
$t_{in}=t+\coi R(t)$.
The scattered wave is then  detected by the receiver
at time $t_{sc}=t+T_{tot}$, with $T_{tot}$ determined implicitly by
\be\label{eqn Ttot}
c_0T_{tot}=R(t)+R\left(t+T_{tot}\right).
\ee
Under   the start-stop approximation, we have $R\left(t+T_{tot}\right)= R(t)$ and thus  $T_{tot}\approx 2\coi R(t)$.
 (This calculation will be refined in Sec. \ref{sec beyond}.)
Using the assumptions above  and   the  Green's function for the wave equation,
\begin{equation}	\label{eq:green}
G(t, {\bm y}) = { \delta(t-|{\bm y}|/c) \over 4 \pi  |{\bm y}|},
\end{equation}
{ convolving $G$ with $f$ to get $E_0$ and then convolving $G$ with the right-hand side of (\ref{linearized wave equation}),
we see that under the stop-start approximation,  the scattered wave measured on the antenna is}

\begin{equation}\label{d(t)}
d(t) := \delta E(t,\gamma(t)) \approx \int_{\R^2} \frac{e^{-i \omega_0 \left(t - 2 R\left(t\right) /c_0\right)}} {\left(4 \pi R\left(t\right)\right)^2} \ p(\bm x,\omega) V(\bm x) d\bm x.
\end{equation}
The {\em amplitude} $p(\bm x,\omega)$ is related to $J_s$ and is the {\em antenna beam pattern}, referring  to the fact that one can phase the individual elements of the antenna to interfere in such a way as to selectively illuminate desired parts of the scene on the ground.
Note that a further refinement is also possible with regard to beam-forming on the received signal \cite{NC02}. For the majority of our discussion, the beam pattern is not important and unless otherwise stated, we take $p\equiv 1$ in (\ref{d(t)}).

Note also that, for ease of notation, we suppress the dependence of $\mathbf R$ and $R$ on $x$.
 Furthermore, the slowly varying range factor in the denominator of (\ref{d(t)}), as well as any constants,
will be absorbed into the amplitude $a$ in \eqref{Wfinal} below.

The Doppler problem, like most radar problems, involves multiple time scales.
The speed of light is $c_0\approx 3 \cdot 10^8$ m/sec, whereas aircraft   speeds are typically subsonic,
or less than about $3 \cdot 10^2$ m/sec.
Even satellites in low earth orbit  travel  only on the order of 8 km/sec $\approx 10^4$ m/sec.
Moreover, the frequencies involved are usually very large, typically above 1 GHz $= 10^9$ Hertz.
To analyze the signal \eqref{d(t)}, we introduce a ``slow time'' $s$, and multiply the data by a windowing function, whose duration is (i) small relative to the antenna motion
(\ie, the distance the antenna travels during the window  is small relative to a wavelength),
but (ii) large enough so that the transmitted signal undergoes a sufficiently  large number of cycles over
the support of the window so as to be
amenable to Fourier analysis.
We take the window about $t=s$ to be of the form $\ell \left(\omega_0\left(t-s\right)\right)$,
where $\ell (t)$ is smooth, identically equal to 1 for $|t|\le L$ and supported in $|t|\le 2L$, for some appropriately chosen $L>0$.
Although of compact support, it is natural to consider
$\ell(\cdot)$  as being a symbol of order zero (see below), as the functions $\ell(\omega_0\cdot)$ are symbols of order 0 uniformly in $\omega_0$ and $L$.

Computing the resulting windowed Fourier transform of the data,
localizing near $s$ in the time variable and using $\omega$ for the frequency variable, we form
\begin{align}	\label{STFT}
W_0(s, \omega) &:= \int e^{i \omega (t-s)} \ell (\omega_0(t-s)) d(t) dt\nonumber\\
&= \int e^{i \omega (t-s)} \ell (\omega_0(t-s)) \int \frac{e^{-i \omega_0 \left(t - 2 R\left(t\right) /c_0\right)}} {\left(4 \pi R\left(t\right)\right)^2}  V(\bm x) d\bm x dt
\end{align}
In \eqref{STFT} we  Taylor expand  $R(t)$  about $t=s$ (cf. \cite[Eqns. (7-8)]{BoCh04})  as
\begin{equation}	\label{Taylor}
R(t) = R(s) + \dot{R}(s) (t-s) + \cdots,
\end{equation}
using the convention throughout  that dot denotes differentiation with respect to time.
Keeping only the linear terms then results in
the approximation of $W_0$ which will be the object of study:
\begin{align}	\label{STFTaylor}
W(s, \omega) &:= \int e^{i \omega (t-s)} \ell (\omega_0(t-s)) \int \frac{e^{-i \omega_0 \left(t - 2 \left(R\left(s\right)
+ \dot{R}(s) (t-s)\right)  /{c_0}\right)}} {(4 \pi R(s))^2}  V(\bm x) d\bm x dt.
\end{align}
After the change of variables $t \mapsto \tau = t-s$, this becomes
\begin{align}\label{Wfinal}
W(s, \omega) &= \int e^{i \omega \tau} \ell (\omega_0\tau) \int \frac{e^{-i \omega_0 (s+\tau - 2 (R (s) + \dot{R} (s) \tau)  /c_0)}}
{(4 \pi R(s))^2}  V(\bm x) d\bm x d\tau \\
& = \int  e^{i \tau (\omega - \omega_0 + 2 \omega_0 \dot{R}/c_0)}
	\underbrace{(\ell (\omega_0\tau) e^{i \omega_0 \left(2 R\left(s\right)/c_0  - s\right)})/((4 \pi R(s))^2)}_{a(s,\omega_0, \bm x;\tau)}
	d\tau\, V(\bm x) d\bm x \nonumber
\end{align}

We define the DSAR transform by $\mathcal F: V(\bm x)\to W(s,\omega)$,
and show below that from the form (\ref{Wfinal}) one can identify $\mathcal F$ as a
{\it Fourier integral operator (FIO)}. Although not stated explicitly in (\ref{Wfinal}), we will implicitly assume that $W$ has been multiplied by a smooth cut-off function with compact support. This makes physical sense but also is necessary to avoid artifacts in the backprojected image later on.
We  will describe the main ideas and results on FIOs and other aspects of microlocal analysis
as needed, but also refer the reader to \cite{Duist,Hor,Hor4} for  more detailed accounts.

\section{The DSAR transform as a Fourier Integral Operator}

A  Fourier integral operator (FIO) is an integral operator whose  kernel has the form
$K(\bm y, \bm x) =  \int e^{i  \phi(\bm y, \bm x,\tau)} a( \bm y, \bm x; \tau) d \tau$,
where the phase $\phi$ and amplitude $a$ satisfy certain conditions outlined below.
In our case, the output variables are $\bm y = (s, \omega)$; the input variables are the components of $\bm x $;
$a$ is as in \eqref{Wfinal}; and $\phi(\bm y, \bm x) = \tau(\omega - \omega_0 + 2 \omega_0 \dot{R}/c_0)$.

\begin{theorem}\label{thm F FIO}
Under assumptions \ref{height_above} and \ref{reflectivity}, the mapping $\mathcal F$ is a Fourier integral operator of order $-1/2$
 associated with the canonical relation $\mathcal C$ in \eqref{eqn Cphi} below.
\end{theorem}

\begin{proof}
To check that $\mathcal F$  given by  \eqref{Wfinal} is an FIO, we need to check the following conditions
on the amplitude and phase of \eqref{Wfinal}.

First, the phase
$\phi(s,\omega,\bm x; \tau) := \tau (\omega - \omega_0 + 2 \omega_0 \dot{R}(s)/c_0)$
is an {\it operator phase function} in the sense of H\"ormander \cite{Hor}, meaning that
its differential in all the variables,
\begin{equation}\label{eqn phasegrad}
d_{s, \omega, \bm x,\tau} \phi = \left( \tau 2 \omega_0 \ddot{R}(s)/c, \tau, 2 \tau \omega_0 d_{\bm x} \dot{R}(s)/c_0, \omega - \omega_0 + 2 \omega_0 \dot{R}(s)/c_0 \right),
\end{equation}
 is nonzero for $\tau\ne 0$.  This is easily checked.

Second, the {\it amplitude}  $a(s,\omega,\bm x;\tau)$   must satisfy certain {\it symbol estimates}.
We note first that  the function  $\ell (\omega_0\tau)$ is a {\it symbol of order 0} in the phase variable $\tau$, which means
that  it  belongs to the class
\begin{equation}	\label{symbolest}
S^0_{1,0}=\big\{ b(s,\omega,\bm x;\tau):
\left|\partial^\alpha_{s,\omega,\bm x}\partial_\tau^j b \right|\le C_{\alpha j}
\left(1+\left|\tau\right|\right)^{-j}\quad \forall j\in \mathbb{N} ,\, \alpha\in \mathbb N^5\big\},
\end{equation}
where $\partial^\alpha_{s,\omega,\bm x} = \partial_s^{\alpha_1} \partial_\omega^{\alpha_2} \partial_{x_1} ^{\alpha_3}  \partial_{x_2}^{\alpha_4} \partial_{x_3}^{\alpha_5}$ and  $\mathbb{N}$ denotes the nonnegative integers.
Moreover, since the amplitude $a$
is a product of a smooth, nonzero
 function of $(s,\bm x)$, independent of $\omega$ and $\tau$,
with the order 0 symbol $\ell (\omega_0\tau)$,
$a$ is also a symbol of order 0, {\it i.e.}, its  derivatives satisfy
\eqref{symbolest} with $\ell (\omega_0\tau)$ replaced by $a(s,\omega,\bm x;\tau)$.

{  Note that, since $\ell(\cdot)$ is of compact support, so is $a$  (as a function of $\tau$).
Thus, $\mathcal F$ is strictly speaking a smoothing operator;
however, in order to understand how singularities in the reflectivity function are transformed into the data, we consider $a$ as a symbol of order 0,
which is uniform in $\omega_0$ and $L$, as mentioned above.
 Note also that $a$ is of compact support in the spatial variables, since we assume that the antenna is compactly supported}.

Finally, the order of the operator comes from the H\"ormander convention
for the orders of FIOs \cite{Hor}:
\begin{align}
\hbox{order}(\mathcal F)&=
\hbox{order}(a)+\frac{\#\hbox{phase vars}}2-\frac{\#\hbox{output vars}+\#\hbox{input vars}}4 \cr
&=0+\frac12-\frac{2+2}4  =-\frac12.
\end{align}
This completes the proof of Theorem \ref{thm F FIO}.

\end{proof}

\medskip
The properties of an FIO involve  the geometry of certain key sets.  The first is called the
  {\it critical manifold} of $\phi$, which  for $\mathcal F$ is defined by
\[Crit_\phi=\left\{(s, \omega, \bm x,\tau): d_{\tau}\phi=\omega - \omega_0 + 2 \omega_0 \dot{R}(s)/c_0 = 0\right\}.\]
Because $\phi$ is nondegenerate in the sense that the full gradient of $d_{\tau}\phi$ in all the variables is nonzero,
it follows that $Crit_\phi$ is a smooth, codimension one surface.

In general, a nondegenerate phase function  determines, in addition to $Crit_\phi$, a second key set,
$\mathcal C $, called the (twisted)
 {\it canonical relation} of $\phi$ (or $\mathcal F$),  which  is a smooth, immersed submanifold of
{ the total cotangent  space of all the variables, both input and output.
The canonical relation describes microlocally
how $\mathcal F$ transforms the locations and directions of singularities in the input variables (i.e., of the reflectivity $V$)
to singularities in the output variables (i.e., of the signal $d$).
Examples of canonical relations include, but are not limited to, the graphs of canonical transformations between the cotangent spaces
of the the input and output variables.}

The twisted canonical relation $\mathcal C$  is  the image of $Crit_\phi$ under the map
$$(s, \omega, \bm x,\tau)\in Crit_\phi  \longrightarrow (s,\omega,d_s\phi,d_\omega\phi; \bm x,-d_{\bm x}\phi)\in T^*\R^2 \times T^*\R^2.$$
 Here $T^*\R^2$ denotes the  cotangent bundle  of $\R^2$,
 which for the purpose of calculations can be considered  as simply  $\R^4$.
With respect to the coordinates
$(s,\omega,\sigma,\Omega;\bm x,\bm\xi)$ on $T^*\R^2\times T^*\R^2$,
$\mathcal C $ can be written as
\begin{align}\label{eqn Cphi}
\mathcal C &= \left\{ (s, \omega, d_s \phi, d_\omega \phi; \bm x, -d_{\bm x} \phi)|_{Crit_\phi} \right \}  \cr
&=\left\{ \left( s, \omega_0 - 2 \omega_0 \dot{R}/c_0 , 2 \omega_0 \tau \ddot{R} /c_0 , \tau; \bm x, \underbrace{-2 \tau \omega_0 d_{\bm x} \dot{R}/c_0}_{\bxi}  \right) : s\in \mathbb{R}, \bm x \in \mathbb{R}^2, \tau \in \mathbb{R} \setminus0  \right\} \cr
&\subset (T^*\R^2\setminus \bm 0)\times  (T^*\R^2\setminus \bm 0).
\end{align}
Here,  $\bm 0$ denotes the {\it zero section}  of $T^*\R^2$, namely $\{(\sigma,\Omega)=(0,0)\}$ and $\bm\xi=\{(0,0)\}$, resp.
{{Since both $\omega_0\ne 0$ and $\tau\ne 0$, the claim above that $\mathcal C $ does not intersect the zero section of either
factor
space follows from Remark \ref{zero_sec} in Sec. \ref{notation} below.}} { From \eqref{eqn Cphi}, one sees that $x, s, \tau$  {
{are coordinates on $C$, which we will use for calculations.}}}
\medskip

In summary,  $\mathcal F$ is an FIO of order $-1/2$ associated with $\mathcal C$,
denoted by \linebreak$\mathcal F\in I^{-\frac12}(\mathcal C)$.
In later sections, we study the  projections from the canonical relation $\mathcal C$ to the two factor spaces $T^*\R^2$ on
the left and right. The left projection $\pi_L: \mathcal C \rightarrow T^*\mathbb{R}^2$ is defined by
\begin{equation}	\label{leftproj}
\pi_L:   (s, \omega, \sigma, \tau; \bm x, \bxi)  \mapsto (s, \omega, \sigma, \tau),
\end{equation}
while the right projection  $\pi_R : \mathcal C  \rightarrow T^*\mathbb{R}^2$ is given by
\begin{equation}	\label{rightproj}
\pi_R:  (s, \omega, \sigma, \tau; \bm x, \bxi) \mapsto (\bm x, \bxi).
\end{equation}
Understanding these projections gives us information about the geometry of $\mathcal C $ and thus the properties of $\mathcal F$.
The next subsection gives the background needed for the subsequent analysis.
\medskip

\subsection{The left and right projections and the Bolker condition}\label{subsec Bolker}

It is a basic aspect of microlocal analysis that, for a general canonical relation, the singularities of the projections $\pi_L$ and $\pi_R$,
in the sense of $C^\infty$ singularity theory \cite{GoGu}, e.g.,
points where their differentials have less than maximal rank,
have important implications for the operator theory of associated FIOs.
In particular, in applications to inverse problems via the study of linearized forward maps,
the singularities of $\pi_L$ and $\pi_R$ determine whether
reconstruction  via filtered backprojection (modulo
smooth errors) is possible,
and allows
the characterization of artifacts
{\it cf.}
\cite{NC04,Fe05,FeGr10,FeQu11,Amb13,FeGaNo13,QuRu13,FeNo15,Amb18}.

 For any canonical relation, say $\mathcal C_0\subset T^*\R^n\times T^*\R^n$
 associated with an FIO $\mathcal F_0$,
 if one of the two  maps $D\pi_L$ and $D\pi_R$
 is nonsingular  at a point $\lambda \in \mathcal C_0 $, then so is the other.
These derivatives  $D\pi_L$ and $D\pi_R$ are each represented by a $(2n)\times (2n)$ matrix,
so the nonsingularity condition takes either of the forms,
\be\label{eqn rightleftnz}
\det(D\pi_L)(\lambda)\ne 0\hbox{ if and only if }\det(D\pi_R)(\lambda)\ne 0.
\ee

We denote by $\Sigma$ the set
  $$\Sigma=\left\{\lambda_0 \in \mathcal C_0: \det(D\pi_R(\lambda_0))=0 \right\}. $$
 If $\Sigma=\emptyset$, we say that  $\mathcal C_0$ is {\it nondegenerate}; otherwise, it is {\it degenerate}.
If  the complement of $\Sigma$ contains a { sufficiently  small}
neighborhood of  $\lambda_0$,  {\it i.e.} the determinant of  \eqref{eqn rightleftnz} is nonzero in a neighborhood of  $\lambda_0$,
then
 $\mathcal C_0$ is a {\it local canonical graph} near $\lambda_0=(y_0,\eta_0,x_0,\xi_0)$
in the sense that  $\mathcal C_0$ is the graph of a canonical transformation
 $\chi:T^*\R^n\to T^*\R^n$ defined near $(x_0,\xi_0)$.
  See, e.g.,  \cite{Hor}, \cite[Thm. 21.2.14]{Hor3}.
  \smallskip

Even though $D\pi_L$ and $D\pi_R$  drop rank on the same set, and their kernels have the same dimension,
 the maps $\pi_L$ and $\pi_R$ may have different types of singularities.
 \smallskip

 Backprojection methods
attempt to form an image by applying the adjoint
$\mathcal F_0^*$ to the data $\mathcal F_0 V$.
If $\mathcal C_0$ is  a  local canonical graph,
the formation of the composition
 $\mathcal F_0^*\mathcal F_0$ is  covered by the {\it transverse intersection calculus}  for FIOs  \cite{Hor,Hor4,Duist},
 resulting in $\mathcal F_0^*\mathcal F_0\in I^{2m}(\mathcal D)$ with $\mathcal D$ a canonical relation
 containing part of the {diagonal} relation,
 $\Delta:=\{(\bm x,\bm \xi, \bm x,\bm \xi): (\bm x,\bm \xi)\in T^*\R^n\setminus 0\}$
 \cite[\S25.2-3]{Hor4}.
If the canonical relation $\mathcal D$ contains only points of the diagonal $\Delta$, then $\mathcal F_0^*\mathcal F_0$
 is a {\it pseudodifferential operator}.  In this case,  {under an illumination assumption,} one can construct a {\it left parametrix}  for
$\mathcal F_0^*\mathcal F_0$, i.e., a pseudodifferential operator $Q$ of order $-2m$ satisfying
$Q\mathcal F_0^*\mathcal F_0=I$ up to a smoothing operator ({ assuming that $F^*F$ is elliptic}).
 This means that $Q \mathcal F_0^*$ is  a filtered backprojection operator that  allows us to reconstruct  a  function $V(\bm x)$  from the
data $\mathcal F_0 V$, up to  a smooth error.
 Any sharp features in $V(\bm x)$ (such as discontinuities or edges)  that are visible in the data will be present in the image,
 and vice versa; we say that there are no {\it artifacts} in the reconstruction.

If, on the other hand,
 $\mathcal D$ contains off-diagonal points, not in the diagonal $\Delta$,
 using straightforward backprojection reconstruction  results in {artifacts},
 i.e., spurious features in the image which are not present in the original scene.
 This is prevented if the Bolker condition \cite{Guill85}  is satisfied, which in the DSAR setting reduces to
\begin{eqnarray}\label{eqn Bolker}
 &(i)& \Sigma=\emptyset \hbox{ (i.e., the projections are nonsingular everywhere);} \cr
 &(ii)& \pi_L \hbox{ is one-to-one (injective).}
\end{eqnarray}

 An injectivity condition on $\pi_L$ such as (ii)  is natural for the prevention of artifacts, since injectivity implies that different points in the scene  correspond (microlocally) to different points in the data, while (i) implies that the standard transverse intersection calculus of FIO applies to the composition $\mathcal F_0^*\mathcal F_0$.
When combined into the Bolker condition \eqref{eqn Bolker},
these ensure that $\mathcal D\subset \Delta$, that
 $\mathcal F_0^*\mathcal F_0$ is a pseudodifferential operator, and that there are no artifacts in reconstructions made using $\mathcal F_0^*$.

\section{Contributions of this paper}

We have shown above that the DSAR operator  $\mathcal F$  is an FIO and consequently can be analyzed by the techniques of microlocal analysis.
In the remainder of this paper, we  analyze the
geometry of the canonical relation $\mathcal C$,
and its effect on the presence (or absence) of artifacts in backprojected images,
by studying the
 geometry of $\mathcal C $, and its
implications for reconstruction of the scene
$V(\bm x)$ via filtered backprojection for two model geometries:
when the transceiver trajectory is either a straight line or a circle,
with   flight path at constant altitude over flat topography.
\smallskip

For the straight-line  trajectory,  we show in \S\ref{sec straight} that $\mathcal C $
is degenerate over the projection of the flight path onto the Earth's surface,
where it has a {\it fold/blowdown} degeneracy,
i.e., $\pi_L$ and $\pi_R$ have singularities of (Whitney) fold  and blowdown types
resp.; descriptions of these singularity classes can be found in Sec. 5.2.
 Our main result
on characterization of $\mathcal F$ for the straight flight path is
\medskip

\begin{theorem}\label{thm line}
If the flight path $\gamma$ is a straight, horizontal line, then the DSAR transform $\mathcal  F$ is an FIO of order $-1/2$, associated to a canonical relation $\mathcal C$  with a fold/blowdown degeneracy. There is a left-right artifact about the projection, $\Gamma$ of the flight path $\gamma$ onto the ground. By suitable beam forming, the artifact can be eliminated, in which case
the data  $\mathcal F V$ determines  any scene $V$ supported away from $\Gamma$, up to a possible   $C^\infty$ smooth error.
\end{theorem}

{{The proof of Theorem \ref{thm line} will be given in \S\ref{sec straight}.}}

\medskip

Canonical relations with fold/blowdown singularities  were studied by Guillemin \cite{GuillBook} and,
in the context  of standard monostatic SAR with a straight flight path and an isotropic antenna pattern,
by Nolan and Cheney \cite{NC02} and Felea \cite{Fe05}
\footnote{{ This is also the transpose of the blowdown/fold geometry that occurs for restricted X-ray transforms \cite{GrUh89}.}}.
Such canonical relations correspond to problems in which
 naive filtered backprojection results in left-right artifacts,
with objects to one side of the flight path appearing in the image on both sides.
Consequently, whenever it is possible, such systems
use side-looking antennas, so that only one of the ambiguous locations is illuminated.
With a side-looking antenna beam pattern, the
Bolker condition is satisfied (which, in this equi-dimensional setting means that $\mathcal C$ is the graph of a canonical transformation $\chi:T^*\R^2\to T^*\R^2$),
and consequently artifact-free reconstruction of the scene
$V(\bm x)$ from the data $W(s,\omega)$ is possible.
This allows for stable reconstruction of the scene $V(\bm x)$ from the data $W(s,\omega)$ by  filtered backprojection,
without any artifacts due to geometry.

\medskip

On the other hand, for a circular flight path  we show in the lengthier analysis in \S\ref{sec circular}   that
the forward operator  $\mathcal F$
has a canonical relation $\mathcal C$  with a more complicated  {\it fold/cusp} degeneracy (i.e., $\pi_L$ has a fold singularity and $\pi_R$ has a cusp singularity).
There is no longer a simple left-right artifact; it is a singular canonical relation, called an open-umbrella (See Remark \ref{rem circular}).
However,
 by an appropriate choice of antenna beam pattern, one can again
restrict the microlocal support
so that the Bolker condition is satisfied, and
this enables stable reconstruction of the scene, up to a smooth error.
We give explicit criteria for portions of the scene that can be imaged in this way.
Our main result for the circular flight path {{is the following, the proof of which  is given in \S\ref{sec circular}.}}

\begin{theorem} \label{thm circle}
If the flight path is circular, the DSAR map $\mathcal  F$ is an FIO of order $-1/2$ associated to a canonical relation $\mathcal C$ with a
fold/cusp degeneracy.
For a scene $V$ with suitable support, or by suitable beam forming, the associated artifact can be eliminated,
and then the data  $\mathcal F V$  determines  $V$ up to a possible smooth  $C^\infty$ error.
\end{theorem}

\medskip

In \S\ref{sec beyond} we modify the  start-stop approximation used
in the main analysis
by adding a first-order correction term,
modeling a nonzero transit time for the wave.
As a result,
the scattered wave is being received at a later time,
and thus a different point along the flight path,
 than when and from where it was transmitted.
We show that the result in Thm. \ref{thm line} characterizing $\mathcal F$ and its canonical relation  for the linear flight path
 under the  start-stop approximation
is stable with respect to this correction, even though the blowdown singularity class
to which $\pi_R$ belongs is not structurally stable and thus is sensitive to general perturbations.
This supports the robustness of the results derived under the  start-stop approximation.
\medskip

We begin, in the next section, with generalities needed for the analysis of  the straight and circular flight paths,
but which would be applicable to other geometries as well.

\section{Notation, key properties, projections and singularities} \label{notation}

We will need properties of the range function $R$ defined in \eqref{def R} and its derivatives.
Recall from \eqref{def R} that  the range vector is $\bm{R}(s):= (\bm x,0)-\bm\gamma(s)$
and the range function is $ R(s):=|\bm{R}(s)|$ (with $\bm x$ suppressed in the notation).
Denote the unit vector $ \bm R(s)/R(s)$ by $\widehat{\bm R}$.
Let $J$ denote the natural inclusion of $\mathbb R^2$ into $\mathbb R^3$,
$J (x_1, x_2) := (x_1, x_2, 0)$,
and $J^*$ its transpose,  which is the projection  $ J^* (\xi_1, \xi_2, \xi_3) := (\xi_1, \xi_2)$.
Finally, we also define the scaled projection,
 \begin{equation}
 P \bm v := \frac{ \bm v - \widehat{\bm R} ( \widehat{\bm R} \cdot \bm v)} {R} = \frac{I - \widehat{\bm R} \otimes \widehat{\bm R}} {R}  \bm v.
 \end{equation}

 With the convention throughout that the dot above a symbol means partial differentiation with respect to $s$,
 one easily verifies the following identities:
 \begin{align}
 \dot{R} &= - \widehat{\bm R} \cdot \dot{\bgamma},\label{eq f1}\\
 d_{\bm x} \dot{R} &= -J^* P \dot{\bgamma}. \label{eq f4} 
 \end{align}

\begin{remark} \label{zero_sec}
Assumption \ref{height_above} implies that $J^*P\dot\gamma \neq 0$, since $\widehat{\bm R}$ is never horizontal,
and therefore from \eqref{eqn Cphi} and \eqref{eq f4} we see that $\bm \xi \neq \bm 0$ in $\mathcal C $.
\end{remark}

\medskip

\subsection{Properties of the right and left projections}\label{subsec rightleft}
The right projection \eqref{rightproj} $\pi_R : \mathcal C \rightarrow T^*(\mathbb{R}^2)$, expressed in terms of the coordinates
$(\bm x,s,\tau)\in \R^3\times (\R\setminus 0)$ on $\mathcal C$ and $(\bm x, \bxi)$ on $T^*\R^2$,
has derivative

\begin{equation} \label{right_deriv}
\renewcommand\arraystretch{1.3}
D_{\bm x, s, \tau} \pi_R = \begin{pmatrix}
\begin{array}{c|c}
  \bm I_{2\times 2}  &
 \bm 0_{2\times 2} \\
  \hline
  \bm 0_{2\times 2}  & \frac{D(\bxi)}{D(s, \tau)}
\end{array}
\end{pmatrix}.
\end{equation}
Recalling that $\bxi = -2\tau\omega_0 d_x \dot{R}/c_0$, we have
(using the parametrization (\ref{eqn Cphi}))

\begin{align}\label{eqn pir}
 \frac{D(\bxi)}{D(s, \tau)} = \frac{-2 \omega_0}{c_0}
	\bigg[ \tau \partial_s d_{\bm x} \dot{R} \ , \ d_{\bm x} \dot{R} \bigg]
= \frac{-2 \omega_0}{c_0}
	\bigg[ \tau d_{\bm x} \ddot{R} \ , \ d_{\bm x} \dot{R} \bigg].
 \end{align}
We  will thus examine conditions under which  $d_{\bm x} \ddot{R}$ and $d_{\bm x} \dot{R} $
 are linearly independent and,  when they are not, analyze the kernel of \eqref{right_deriv}.
 \smallskip

For the Bolker condition \eqref{eqn Bolker},  it is  important for us to understand the injectivity (or its failure) of  $\pi_L$, which reduces to the injectivity of
 \begin{equation}
 \bm x \mapsto \left( \omega_0 (1 - 2 \dot{R}/c) , 2 \omega_0 \tau \ddot{R}/c_0 \right).
 \end{equation}
This in turn reduces to the question of the injectivity of the map
 \begin{equation}
 \bm x \mapsto (\dot{R}, \ddot{R}).
 \end{equation}
For two points $\bm x,\, \bm y\in\mathbb{R}^2$, labeling  the associated quantities with subscripts for clarity,
the condition $\dot{R}_{\bm x} = \dot{R}_{\bm y}$ is easily understood as the Doppler condition.
By \eqref{eq f1}, it says that the down-range relative velocity $\dot{R}_{\bm x} =
-\widehat{\bm R}_{\bm x} \cdot \dot{\bgamma}$ to $\bm x$ is the same as that to $\bm y$,
 or alternatively that the unit vectors $\widehat{\bm R}_{\bm x}$ and $\widehat{\bm R}_{\bm y}$
 lie on the same circle in the plane perpendicular to $\dot{\bgamma}$,
 or alternatively, that $\bm x$ and $\bm y$ must lie on the same circular cone with vertex $\bgamma$ and axis $\dot{\bgamma}$.
 On the other hand, the condition $\ddot{R}_{\bm x} = \ddot{R}_{\bm y}$ seems harder to characterize.
 \smallskip

In the next two sections we  focus on    two particular  flight trajectories, either straight  or circular.
For simplicity,  as noted above we will  assume that the flight path
is at a constant altitude over a flat landscape, although that is not necessary for the application of the microlocal approach.
\medskip

\subsection{Singularities of smooth maps}\label{sec appendix}

For the convenience of the reader, we  summarize the needed definitions of  singularity classes;
{  see, e.g., \cite{AGV,BroLan,GoGu} for  more detailed treatments of singularity theory of smooth functions.}
Let $M$ and $N$ be manifolds of dimension $n$; $f:N\rightarrow M$ be a $C^{\infty}$ function; and define $\Sigma\subset M$,

\[\Sigma:=\left\{x\in N\:|\: \det (Df(x))=0\right\}.\]
For the classes considered, we make the basic assumption that $d\left(\det \left(Df\left(x\right)\right)\right)\ne\nolinebreak0$,
so that (i) $\Sigma$ is a smooth hypersurface, and (ii) at points of $\Sigma$, $dim\left(\ker Df \right)=1$.
There then exists a (nonunique) {\it kernel} vector field, i.e., a nonzero vector field $V$ along $\Sigma$
such that $Df(p)\left(V\left(p\right)\right)=0$ for all $p\in\Sigma$.
A map $f$ satisfying these conditions is said to have a {\it corank one} singularity.
We now define some classes of corank one singularities.

\begin{definition}
$f$ is said to have a  \textit{(Whitney) fold} singularity along $\Sigma$ if it only has corank one singularities and, in addition,   for every $p\in\Sigma$, $\ker Df(p)$ intersects $T_{p}\Sigma$ transversally.
Equivalently, $\langle d(\det Df), V\rangle\ne 0$ on $\Sigma$.
\end{definition}
\smallskip

\begin{definition}
$f$ is said to have a \textit{blowdown} singularity along $\Sigma$
if it only has corank one singularities and, in addition,
$\ker Df(p)\subset T_{p}\Sigma$ for every $p\in\Sigma$.
\end{definition}
\smallskip

\begin{definition}\label{def cusp}
$f$ is said to have a \textit{cusp} singularity along $\Sigma$
if it only has corank one singularities and, in addition,
at any point $p\in\Sigma$ where $\langle d(\det Df), V\rangle= 0$,
i.e., where it fails to be a fold, one has
$\left\langle d \left\langle d\left(\det Df, V\right\rangle\right),V\right\rangle\ne 0$.
\end{definition}
\smallskip

More precisely, for a cusp, let $f(x_1, x_2, \dots, x_N)=(x_1, x_2, \dots, h(x))$ with $h(0)=\nolinebreak0$.
Then, $\Sigma_1=\{x: \frac{\partial h}{\partial x_N} (0)=0 \}$ and $f$ has a cusp singularity
at $0$ if $\frac{\partial^2 h}{\partial x_N^2} (0)=\nolinebreak0,\, \frac{\partial^3 h}{\partial x_N^3} (0) \neq 0$
and if rank $\big[\left. d_x(\frac{\partial h}{\partial x_N}) \right|_{x=0}, \left. d_x(\frac{\partial^2 h}{\partial x_N^2})\right|_{x=0} \big]=2$.
In other words, Ker$(d f)=\R\cdot\frac{\partial}{\partial x_N}$ is tangent simply to $\Sigma_1$
along $\Sigma_{1,1}(f)=\{ \frac{\partial h}{\partial x_N} (0)=\frac{\partial^2 h}{\partial x_N^2} (0)=0 \}$, and the gradients of $\frac{\partial h}{\partial x_N}$ and $\frac{\partial^2 h}{\partial x_N^2}$ are linearly independent at $x=0$.

 \section{The case of a linear flight path}\label{sec straight}

Without loss of generality, the  trajectory of a straight flight path at height $H>0$ and of constant, unit speed
can be assumed to be $\bgamma(s) = (s, 0, H)$.
For the straight flight path, $R$ becomes $R=\sqrt{(x_1-s)^2+x_2^2+H^2}$ and the derivatives with respect to $s$ become

 \begin{align}
 \dot{R} &= - \frac{x_1 - s }{R}, \label{straight 1}\\
 \ddot{R} &=  \frac{x_2^2 + H^2}{R^3}. \label{straight 2}
  \end{align}

\begin{proof}[{{{\it Proof of Theorem \ref{thm line}}}}]

{{We have}}

\begin{align}
\nx R & = R^{-1}\left(x_1-s,x_2\right),\label{eqn nxR}\\
  d_{\bm x}\dot{R} & = R^{-3}\left( - (x_2^2 + H^2), (x_1 - s) x_2 \right), \label{straight 4} \\
  d_{\bm x}\ddot{R} &= R^{-5}\left( -3(x_1 - s) (x_2^2 + H^2) , \  2 x_2 R^2 - 3 x_2 (x_2^2 + H^2) \right).\label{straight 5}
 \end{align}
 Computing the determinant of the right projection $\pi_R:  (x_1, x_2, s, \tau) \mapsto (x_1, x_2, \xi_1, \xi_2)$,
 from (\ref{eqn pir}) we obtain
 \begin{eqnarray*}
 \det \left(\frac{D(\xi_1, \xi_2)}{D(s,\tau)}\right) = \left(\frac{4\omega_0^2}{c_0^2}\right) \det (d_{\bm x}\dot{R}, d_{\bm x}\ddot{R} )
 = \left(\frac{4\omega_0^2}{c_0^2R^{6}}\right)x_2(x_2^2+H^2).
  \end{eqnarray*}
 Thus, the differential  $D\pi_R$ drops  rank  by one along the hypersurface
 $$\Sigma =\{(s,x_1, x_2, \tau)\in\mathcal C:   x_2=0; \  x_1, s, \hbox{ and } \tau\ne 0  \hbox{ arbitrary }\}, $$
 and drops rank
 simply in the sense that $d\left(\det\left(D\pi_R\right)\right)\ne 0$ at $\Sigma$.
 We note that $\Sigma$ is the set of
points of $\mathcal C$ above the line $\Gamma$,
 the projection of the flight path on to the ground plane.
 In addition, with an isotropic  antenna beam pattern, the entire problem is invariant with respect to reflection about
 the plane $x_2=0$, leading (as we
will see) to a left-right artifact about $\Gamma$ in reconstructions of the scene.
 \smallskip

To classify the type of singularity of $\pi_R$ on $\Sigma$,
one sees from \eqref{right_deriv}, \eqref{eqn pir} that
the kernel of $D\pi_R$, necessarily  spanned by a linear combination of the vector fields
$\partial_s, \partial_{x_1},\partial_{x_2}$ and $\partial_\tau$,
has no $ \partial_{x_1},\partial_{x_2}$ components and so is in fact spanned by a combination of only
$\partial_s$ and $\partial_\tau$,
both of which are tangent to $\Sigma$; thus,
 $\ker(D\pi_R)$ is tangent to $\Sigma$ everywhere.
 Thus (see \S 5.2), $\pi_R$ has a {\it blowdown} singularity at $\Sigma$.
 \smallskip

Similarly, for $\pi_L$, from \eqref{eqn Cphi} one computes that the $\partial_s$ and $\partial_\tau$ components of any vector in
$\ker(D\pi_L)$ must be zero. The kernel is the null space  in the $x_1,x_2$ variables, of $\left[-d_{\bm x}\dot{R}, \tau d_{\bm x}\ddot{R}\right]^T$. Evaluating \eqref{straight 4} and \eqref{straight 5} at $x_2=0$, we see that the second column of this matrix is zero, and thus $\ker(D\pi_L)= span\{   \partial_{x_2} \}$.
This is transversal to $\Sigma$; hence, $\pi_L$  has a {\it fold} singularity at $\Sigma$  (cf. \S 5.2).
{ Thus, the canonical relation $\mathcal{C}$ is
 a fold/blowdown canonical relation as discussed in \S4}.
This {{concludes}} the proof of Thm. \ref{thm line}.
\end{proof}

To summarize: In the case of a straight flight path,
the forward map $\mathcal F$ taking the scene $V$ to the windowed DSAR data $W$
 is   an FIO, given by \eqref{Wfinal},  associated with a fold/blowdown canonical relation $\mathcal C $.
 This means that, as discussed above,
 without beam forming to one side of the flight path or the other,
backprojection will  potentially create left-right  artifacts in the image which are just as strong
 as the bona-fide part of the image.
\begin{remark}\label{rem linear}
As in Felea \cite{Fe05} it can be shown  that $\mathcal F^* \mathcal F \in I^{-1,0}(\Delta, \mathcal C_\chi)$ where
 $\Delta\subset T^*\R^2\times T^*\R^2$ is the diagonal and
 $\mathcal C_\chi\subset T^*\R^2\times T^*\R^2$
 is an artifact relation, given by the graph of the
canonical transformation
 $\chi(\bm x, \bm \xi):=(x_1, -x_2, \xi_1, -\xi_2 )$.

 Here, $I^{p,l}(\Delta,\Lambda)$ is the class of {\it pseudodifferential operators
with singular symbols} introduced in \cite{MeUh79,GuUh81}
now usually referred to as {\it paired Lagrangian operators}.
As a result, an image extracted from $\mathcal F^* \mathcal F$ can have left-right artifacts
just as strong as the actual features being imaged, i.e., the order of $F^*F$ is the same on $\Delta$ and $\Lambda$.
It is also possible to reduce the strength of these artifacts using a filtered backprojection method,
with the principal symbol of the filter vanishing on  $\Sigma$,
along the lines of \cite{FeQu11,QuRu13}, but we will not pursue this here.

  If, on the other hand, the system uses an antenna beam that illuminates only a region lying entirely to the left or to the right of the flight path,
then $\mathcal F$ is an FIO associated with
a canonical graph, which is
a canonical relation satisfying the Bolker condition;
consequently backprojection produces an image without artifacts.
 \end{remark}

\section{The case of a circular flight path}\label{sec circular}

{{The proof of Theorem \ref{thm circle} requires some preliminary discussion.}}

\subsection{Preliminaries}\label{subsec prelim}

For simplicity, we consider the case of the flight trajectory being a circle of radius $\rho>0$ at constant altitude in $\R^3$
and centered above the origin in $\R^2$, parametrized by
$\bgamma(s) = (\rho\cos s , \rho\sin s, \rho h)$.
 Note that, in this case,  we write the height $H$ as $H=\rho h$,
where $h$
 is a dimensionless parameter (see Fig. \ref{fig:theta_c}).

We write $\bm e(s) = (\cos s, \sin s)$ and note that $\bm e^\perp (s) :=  \dot{\bm e}(s) = (-\sin s, \cos s)$.
This yields
{
\begin{align}
R &=  \sqrt{|\bm x - \rho\bm e(s)|^2 + \rho^2h^2},\label{eq circular1}\\
\dot{R} &=
-\rho \frac{\bm x \cdot \bm e^\perp} {R},\label{eq circular2} \\
\ddot{R} &= \rho \left(\frac{\bm x \cdot \bm e}{R}- \rho \frac{(\bm x \cdot \bm e^\perp)^2}{R^3}\right).\label{eq circular 3}
\end{align}}
 From  (\ref{eqn phasegrad}) we have, on the critical set,
\begin{align}\label{eqn omegaeta}
\omega &= \omega_0 + 2\rho \omega_0 \bm x \cdot \bm e^\perp /(c_0R), \cr
\eta &= -  \frac{2\rho \omega_0 \tau}{c_0} \partial_s \left( \frac{\bm x \cdot \bm e^\perp}{R}\right)
=   \frac{2 \rho \omega_0 \tau}{c_0}  \left(  \frac{\bm x \cdot \bm e}{R} -\rho \frac{\left( \bm x \cdot \bm e^\perp \right)^2}{R^3} \right) \ .
\end{align}

To determine whether artifacts can be avoided in the backprojected image, we consider the left projection,
$\pi_L: \mathcal C \to T^*\R^2$.
 Let $\bm x$ and $\bm x'$  correspond to the same $s$ and $\omega$.
 We consider the case when $\omega = \omega'$  and $\eta=\eta'$. If $R = R'$, the fact that $\omega = \omega'$ implies that
$\bm x \cdot \bm e^\perp = \bm x' \cdot \bm e^\perp$, and $\eta = \eta'$
implies that $\bm x \cdot \bm e = \bm x' \cdot \bm e$. Thus  $\bm x = \bm x'$.

\begin{figure}[h]
\centering
\includegraphics[width=0.4\linewidth]{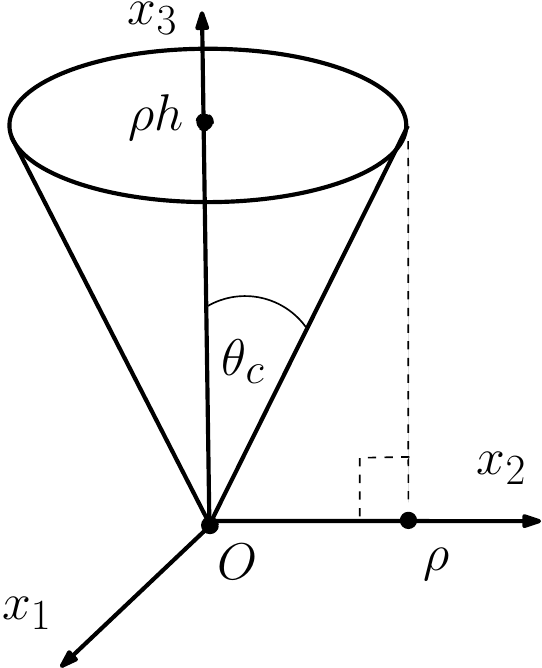}
\caption{Schematic of a circular flight path.}
\label{fig:theta_c}
\end{figure}

\par Hence if there is an artifact for the circular trajectory case, that artifact must have $R \neq R'$.   (As it turns out, neither are they at points that are inverted with respect to the circle.)
This is in contrast to  standard monostatic SAR, since in that case $R=R'$.
\begin{remark} \label{rem combine}
The above shows that any artifacts arising from Doppler imaging for a circular flight path must appear in a location which is different from their location in standard monostatic SAR for the same flight path.  Therefore, in principle, it might be possible to combine traditional monostatic SAR with Doppler SAR imaging in order  to identify and remove artifacts.
\end{remark}

In the next subsection, we investigate if it is possible to localize any backprojected artifacts to be outside the region of interest (ROI).

\subsection{Condition for $\mathcal C$ to be a local canonical graph}
  From the discussion in \S\ref{subsec Bolker}, we know that if $\mathcal C$
  satisfies the Bolker condition \eqref{eqn Bolker}
  then backprojection results in an artifact-free image.
We initially consider the first part of the Bolker condition, namely the requirement that $\mathcal C$
is a local canonical graph, and need to determine where
the derivative of the left projection $\pi_L$  has full rank,
i.e., $\hbox{rank}(D\pi_L)=4$.
We may parametrize $\mathcal C$ using coordinates $(s,\tau, \bm x)$, with respect to which

\begin{eqnarray}
\pi_L(s,\tau,\bm x) = (s,\omega_0(1-2\dot R/c_0),2\omega_0\tau \ddot R/c_0,\tau) \ .
\end{eqnarray}
Next, to  make the study of $\pi_L$ easier, introduce a new, $s$-dependent   coordinate system in the plane,
 $(u,v)$, defined by

\begin{eqnarray} \label{decomp}
 \bm x-\rho \bm e =\rho hS \bm e+\rho h \frac{u}{\sqrt{1-u^2}}C \bm e^\perp,
\end{eqnarray}

where
\begin{eqnarray}\label{def SC}
S=S(v) = \sinh(v),\nonumber\\
C=C(v)= \cosh(v).
\end{eqnarray}
 Note that for $s,\, \bm x$ ranging over compact sets, $u$ satisfies $|u|\le 1-\epsilon$ for some $\epsilon>0$,
and thus $1-u^2$ is bounded away from 0.
The $(u,v)$ coordinates are closely related to elliptic cylindrical  coordinates, but in the plane and centered directly below the antenna;
cf. \cite[\S2.7]{Arf}.

\begin{figure}[htbp]
\begin{center}
\includegraphics[width = 6cm] {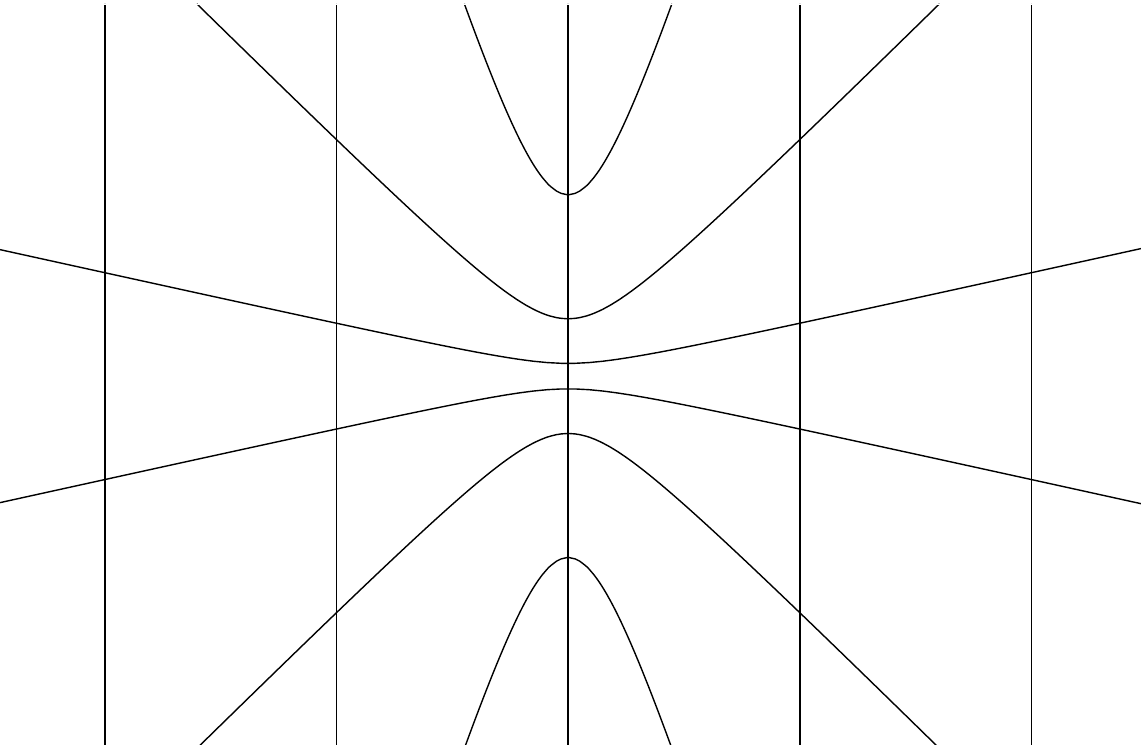}
\caption{ The curves of constant $u$ (hyperbolas) and constant $v$ (vertical lines) for a location on the flight path in which the
flight velocity vector is along the vertical axis.  The coordinate system is centered directly under the antenna.  }
\label{uvcoords}
\end{center}
\end{figure}

We note  that since $\bm e$ is perpendicular to $\bm e^\perp$, we have immediately
\begin{equation} 	\label{ecomponent}
(\bm x-\rho \bm e ) \cdot \bm e = \rho h S.
\end{equation}

The new coordinate system is better understood once we have calculated $\dot R$ in the new coordinates,
as follows. From \eqref{eq circular2} and \eqref{decomp} we have
\begin{eqnarray}	\label{dotReq}
\dot R =  -\rho^2 \frac{huC}{R\sqrt{1-u^2}},
\end{eqnarray}
where
\begin{eqnarray}
R^2 = \rho^2 h^2S^2+\frac{\rho^2h^2u^2C^2}{1-u^2} + \rho^2 h^2 , \nonumber
\end{eqnarray}
so that
\begin{align}	\label{Req}
R &= \rho h \sqrt{ (1 + S^2)+\frac{u^2C^2}{1-u^2}  } = \rho h \sqrt{ \frac{(1 - u^2) C^2+u^2C^2}{1-u^2}  }
= \frac{\rho hC}{\sqrt{1-u^2}},
\end{align}
where we have used $C^2 - S^2 = 1$.
Hence, comparing \eqref{Req} with \eqref{dotReq}, we find
\begin{equation}	\label{dotRu}
\dot R = -\rho u.
\end{equation}
Thus the coordinate $u$ is simply proportional to the Doppler shift.
Because $\dot{\bgamma} = (\rho \bm e^\perp, 0) $, from comparing \eqref{dotRu} with \eqref{eq circular2}, we see that
$ u = \widehat{\bm R} \cdot (\bm e^\perp,0)$, which is clearly bounded in magnitude by $1$.
The set $u = \pm 1$ corresponds to the line on the ground directly under the tangent to the flight path at $\bm e(s)$, and we exclude this by keeping the antenna beam pattern away from  the direction of travel (forward and backward).

Observe that we may also parametrize $\mathcal C$ with the coordinates $(s,\tau,u,v)$. To see this, one needs to check that
\begin{equation} \label{cov 1}
\left| \frac{\partial (x_1,x_2)}{\partial (u,v)} \right|  =
\rho^2 \bigg| \begin{pmatrix} a \bm e^\perp\, , &  (b \bm e+d \bm e^\perp) \end{pmatrix} \bigg| \neq 0,
\end{equation}
where
\begin{equation}\label{cov 2}
a:=hC(v) \frac{1}{(1-u^2)^{3/2}};\quad b:=hC(v); \quad d :=  hS(v) \frac{u}{\sqrt{1-u^2}}.
\end{equation}
\eqref{cov 1} clearly  holds true since $h>0$ means that both $a$ and $b$ in \eqref{cov 2} are nonzero.
\medskip

Therefore, by avoiding data from the  forward and backward directions
(or, alternatively, filtering out echoes associated to values of $\dot R$ near $\pm\rho$),
we have that $(s,\tau,u,v)$  forms a valid coordinate system on $\mathcal C$.
 The coordinate system $(s,\tau,u,v)$  is designed to make any degeneracy of the projection $\pi_L$ appear in a single variable,
 namely $v$. In fact, in terms of $(s,\tau,u,v)$, one has
\begin{eqnarray} \label{left_uv}
\pi_L(s,\tau,u,v) = (s,\omega_0(1+2u \rho/c_0),2\omega_0\tau \ddot R(u,v)/c_0,\tau),
\end{eqnarray}

To classify the singularities of $\pi_L$, one needs to find first the set where  $D\pi_L$  drops rank, i.e, where
\begin{eqnarray*}
\frac{\partial \ddot R}{\partial v} = 0 \ .
\end{eqnarray*}

Computing $\ddot R$ in the original coordinates $x_1, x_2$ first,
and then rewriting it in the new coordinates $u,v$ using  (29) and (30) one obtains

\begin{eqnarray}
\ddot R  &=& \rho \frac{\bm x\cdot \bm e}{R}-\rho^2\frac{(\bm x\cdot \bm e^\perp)^2}{R^3} \nonumber \\
&=&\rho (1+hS)\frac{\sqrt{1-u^2}}{hC}-\rho\frac{h^2u^2C^2}{1-u^2}\cdot \frac{(1-u^2)^{3/2}}{h^3C^3}, \nonumber
\end{eqnarray}
which leads to
\begin{equation}\label{rddot}
\ddot R = \rho \frac{\sqrt{1-u^2}}{h} \left( \frac{1+hS-u^2}{C} \right)\ .
\end{equation}
Therefore, setting the $v$-derivative of \eqref{rddot} equal to zero shows that
the set  $\Sigma$ of points where $\pi_L$ drops rank is given by
\begin{eqnarray} \label{sigma_1}
\Sigma=\left\{\, (s,\tau,u,v)\in\mathcal C:\,  f(u,v) := h^2+(u^2-1)S h= 0\, \right\}.
\end{eqnarray}

\begin{remark} \label{no-zero}
In $(u,v)$ coordinates, $\bf x=\bf 0$ corresponds to $(u_0,v_0)=(0,\sinh^{-1}(-1/h))$, where we have used (\ref{decomp}). Since $f(u_0,v_0)=h^2+1>0$, we see that the fiber of $\Sigma$ lying over $\bf x=\bf 0$ is empty.
\end{remark}

In particular,
$D\pi_L$  is of full rank %
 if and only if
 $f \neq 0$, where $f$  is defined as in \eqref{sigma_1}.
Since $Sh=\frac{(\bm x-\rho \bm e)\cdot \bm e}{\rho}$, we have $f>0$  iff
$$(1-u^2)  (\bm x - \rho \bm e)\cdot \bm e  < \rho h^2.$$
A sufficient condition for this is
\begin{equation}\label{dop_cond1}
 (\bm x - \rho \bm e)\cdot \bm e < \rho h^2,
\end{equation}
and a sufficient, $\bm e$-independent condition for this to hold is that
\begin{eqnarray*}
\bm x\in \mathcal D_{h,\rho} := \left\{ \ \bm x\in \mathbb R^2\ :\ |\bm x| <  \rho (h^2+1)  \right\}.
\end{eqnarray*}

Summarizing our analysis so far:
Suppose that $\Sigma$  is not in the microlocal support of $\mathcal F$ (for example, suppose
 that beam forming ensures that only points in $\mathcal D_{h,\rho}$ are illuminated).
Then the  canonical
relation  of $\mathcal F$  is a local canonical graph.

\smallskip

\begin{remark}
If the flight path only consists of a portion of the full circle, then equation (\ref{g_func}) below could be used to enlarge the region $\mathcal D_{h,\rho}$ where the canonical relation of $\mathcal F$ is a local canonical graph. We do not pursue this idea further here.
\end{remark}

Instead, in the next three subsections we analyze the structure of $\mathcal C$ near $\Sigma$.  It will be convenient to introduce another set of coordinates as follows:
\begin{eqnarray} \label{def_p}
p := (\bm x -\rho\bm e)\cdot \bm e/\rho=hS,\\
q := (\bm x-\rho \bm e)\cdot \bm e^\perp/\rho=\bm x\cdot \bm e^\perp/\rho.
\end{eqnarray}
Observe that $(s,\tau,p,q)$ also form a coordinate system on the canonical relation $\mathcal C$.  Indeed, since we have already remarked that $(s,\tau,x_1,x_2)$ are coordinates on $\mathcal C$, one only has to notice that for a fixed $(s,\tau)$-value, the map $(x_1,x_2)\mapsto (p,q)$ is a diffeomorphism since the vectors $d_{\bm x}p=\bm e/\rho$
and $d_{\bm x}q=\bm e^\perp/\rho$ are linearly independent. Points on $\Sigma$ satisfy
\begin{eqnarray*}
h^2+(u^2-1)Sh=0,
\end{eqnarray*}
which  from \eqref{Req}  leads to
\begin{equation*}
h^2-\frac{\rho^2h^2C^2}{R^2}p =0,
\end{equation*}
 which, with $C^2 - S^2 = 1$, leads to

\begin{equation*}
h^2R^2-\rho^2(h^2+h^2S^2)p = 0,
\end{equation*}
and hence
\begin{equation}
h^2R^2-\rho^2(h^2+p^2)p = 0.
\end{equation}
Therefore
\begin{eqnarray} \label{g_func}
g(s,x) :=h^2R^2(s,x)-\rho^2(h^2+p^2(s,x))p(s,x) = 0 \label{sigma_rewrite}\
\end{eqnarray}
is a defining equation for $\Sigma$.
 Since $R^2 = \rho^2 (p^2 + q^2 + h^2)$,  \eqref{g_func}
 can also be re-written as
\begin{eqnarray*} \label{g_tilde}
h^2\rho^2(h^2+p^2+q^2)-\rho^2(h^2+p^2)p=0.
\end{eqnarray*}
Hence
\begin{equation}\label{cubic}
\tilde{g}(p,q):=p^3-h^2(p^2+q^2)+h^2p-h^4=0
\end{equation}
is also a defining equation for $\Sigma$. We will next make use of $\tilde g$ to analyze the properties of $\Sigma$.\\

{{{\it Proof of Theorem \ref{thm circle}.}}}
Due to the length of the proof,  it will be presented in parts: first we consider the singularity of $\pi_L$ (Sec. \ref{subsec piL}), 
then the singularity of $\pi_R$ (Sec. \ref{subsec piR}) and finally the elimination of artifacts (Sec. \ref{subsec absence}).

\subsection{The singularity of $\pi_L$}\label{subsec piL}
We showed in the last section that $D\pi_L$ drops rank at $\Sigma$. The following lemma establishes a key property of $\Sigma$.

\begin{lemma} \label{L1}
The  function $\tilde g$ in \eqref{cubic} is a defining function for $\Sigma$:  $d\tilde g\ne0 $ at all points of $\Sigma$.
\end{lemma}
\begin{proof}  
Clearly $(\partial\tilde g/\partial q)(p,q) \neq 0$ when $q\neq 0$, so we just need to check $(\partial\tilde g/\partial p)(p,0) \neq 0$ whenever $\tilde g(p,0)=0$.
To see this, we argue as follows: note that
\begin{eqnarray*}
\frac{\partial \tilde g}{\partial p} (p,q) = 3p^2-2h^2p+h^2.
\end{eqnarray*}
Recall that $p=hS$ and suppose for contradiction that $(\partial \tilde g/\partial p)(p,0)=\tilde g(p,0)=0$ for some $p$.  Then
\begin{eqnarray}
\frac{\partial \tilde g}{\partial p}(p,0)=0 \Rightarrow 3(hS)^2-2h^2(hS)+h^2=0 \nonumber \\
\Rightarrow 3S^2-2hS+1=0 \label{g1}
\end{eqnarray}
and
\begin{eqnarray}
\tilde g(p,0)=0 \Rightarrow (hS)^3-h^2(hS)^2+h^3S-h^4=0 \nonumber \\
\Rightarrow hS^3-h^2S^2+hS-h^2=0. \label{g2}
\end{eqnarray}
Subsitituting (\ref{g1}) into (\ref{g2}) gives
\begin{eqnarray*}
hS^3-h^2S^2+hS+3h^2S^2-2h^3S=0\\
\Rightarrow S^2+2hS+(1-2h^2)=0,
\end{eqnarray*}
where we have cancelled a factor of $hS$, which is valid since (\ref{g1}) tells us that $S\neq 0$.  The roots of this quadratic equation are
\begin{eqnarray}
S=-h\pm \sqrt{3h^2-1}. \label{g3}
\end{eqnarray}
Recalling that for points in $\Sigma$ we have
\begin{eqnarray*}
(1-u^2)S=h
\end{eqnarray*}
and substituting this into (\ref{g3}) gives
\begin{eqnarray}
(1-u^2)(-h\pm \sqrt{3h^2-1})=h, \label{g4}
\end{eqnarray}
which can only be true if we take the plus sign (since $0<h, 0<1-u^2$).  Simplifying (\ref{g4}), we get
\begin{eqnarray*}
\sqrt{3h^2-1}=h+\frac{h}{1-u^2},\\
\end{eqnarray*}
hence
\[h^2\left\{3-\left(\frac{2-u^2}{1-u^2}\right)^2\right\}=1,\]
which leads to
\[\frac{3(1-u^2)^2-(1+(1-u^2))^2-1}{(1-u^2)^2}=1,\]
\begin{equation}\label{final lemma 1}
\frac{2(1-u^2)(-u^2)-1}{(1-u^2)^2}=1.
\end{equation}
This leads to a contradiction since the left-hand side of \eqref{final lemma 1} is strictly negative and it completes the proof of the Lemma. \end{proof}
Note that Lemma \ref{L1} shows that $\Sigma$ is a smooth  hypersurface   in $\mathcal C$.

\smallskip

If we illuminate a neighborhood of the boundary $\partial\mathcal D_{h,\rho}$, then it is of interest to know what kind of singularity $\pi_L$ has at $\Sigma$.  This is answered by the next lemma.
\begin{lemma} \label{L2}
The projection $\pi_L$ has a \textit{fold} singularity at $\Sigma$.
\end{lemma}
\begin{proof} Note that, at points of $\Sigma$,   $\ker(D\pi_L)=\textnormal{span}\{\partial_v\}$
and we saw from the proof of Lemma \ref{L1} that $d\tilde g\neq 0$ there as well, so $\pi_L$ drops rank simply at $\Sigma$. Moreover,
\begin{eqnarray*}
 \left. \frac{\partial f}{\partial v}\right|_{\Sigma} = -(1-u^2)hC \ne  0.
\end{eqnarray*}
It follows that $\ker(D\pi_L)$ intersects $T\Sigma$ transversally,  proving the lemma.
\end{proof}

 \subsection{The singularity of $\pi_R$}\label{subsec piR}

Since $D\pi_L$ drops rank by 1 at $\Sigma$, the Thom-Boardman notation of singularity theory suggest relabeling $\Sigma $ as $\Sigma_1$
\cite{GoGu}.
Then, equality (\ref{right_deriv}) shows that $\ker(D\pi_R)$ at points of $\Sigma_1$, denoted $\ker(D\pi_R|_{\Sigma_1})$ is spanned by a linear combinations of $\partial_s$ and $\partial_{\tau}$.
Since $g$ does not depend on $\tau$, it makes sense to  investigate the degree to which $g$ vanishes with respect to $s$ at $\Sigma$. Noting that $\dot p=(\bm x\cdot \bm e^\perp)/\rho$, we have
\begin{eqnarray}\label{gs}
g_s &=& 2h^2R\dot R -\rho(h^2+p^2)(\bm x\cdot \bm e^\perp)-2\rho p^2(\bm x\cdot \bm e^\perp) \nonumber \\
 &=&-2\rho h^2(\bm x\cdot \bm e^\perp) -\rho(h^2+p^2)(\bm x\cdot \bm e^\perp)-2\rho p^2(\bm x\cdot \bm e^\perp)  \nonumber \\
&=& -3\rho (h^2+p^2) (\bm x\cdot \bm e^\perp)
\end{eqnarray}
Hence, if $g(s,\tau, \bm x)=0$, then $g_s(s,\tau, \bm x)=0$ iff $\bm x\cdot \bm e^\perp=0$, i.e., when $\bm x$ lies on the line directly to the left or right of the transceiver.
Thus, again using the Thom-Boardman notation, we investigate the properties of
\begin{eqnarray*}
\Sigma_{1,1} := \left\{(s,\tau, \bm x)\in \Sigma_1\ : \bm x \cdot \bm e^\perp = 0 \ \right\} = \Big\{(s,\tau,p,q)\in \mathcal C\ : \tilde g(p,q)=0, q=0\ \Big\}
\end{eqnarray*}
as follows.

\begin{lemma} \label{L3}
$\Sigma_{1,1}$ is a smooth, codimension one, immersed submanifold of $\Sigma_1$.
\end{lemma}
\begin{proof}
We saw in the proof of Lemma \ref{L1} that $\tilde g(p,0)=0$ implies $\frac{\partial\tilde g}{\partial p}(p,0)\neq 0$ and since $\tilde g$ is smooth, the lemma follows.
\end{proof}

\begin{lemma}
The projection $\pi_R$ has a \textit{cusp} singularity at $\Sigma_{1,1}$.
\end{lemma}
\begin{proof} The non-vanishing of $d \tilde g|_{\Sigma_{1,1}}$ means that $\pi_R$ drops rank by one,
with $\det(D\pi_R)$ vanishing simply at $\Sigma_1$. Moreover, ker$(D\pi_R|_{\Sigma_{1,1}})\subset T\Sigma_1$.
The vector field $\partial_s|_{\Sigma_1}$ is tangent to $\Sigma_1$ at points of $\Sigma_{1,1}$,
since $g$ and $g_s$ are defining functions for $\Sigma_1$ and $\Sigma_{1,1}$ resp.
Suppose that $g_{ss}(s,\tau, \bm x)=0$ for some point $(s,\tau, \bm x)\in \Sigma_{1,1}$.
Then equation (\ref{gs}) implies $\bm x\cdot \bm e= \bm x\cdot \bm e^\perp=0$, meaning that $\bm x= \bm 0$.
This contradicts Remark \ref{no-zero} since $\Sigma_{1,1}\subset\Sigma_1=\Sigma$.
Therefore $dg(\partial_s)$ has a simple zero at $\Sigma_{1,1}$.
In $(p,q)$ coordinates, $g_s$ becomes $\tilde g_s=-3\rho^2q(p^2+h^2)$.  We need to check that $d \tilde g|_{\Sigma_{1,1}}$ and $d \tilde g_s|_{\Sigma_{1,1}}$ are linearly independent. One has
\begin{eqnarray*}
d g(p,0) = (\partial\tilde g/\partial p(p,0),0)\ ;\quad d\tilde g_s(p,0)  = (0,-3\rho^2(h^2+p^2))
\end{eqnarray*}
However,  we already saw in the proof of Lemma \ref{L1} that $\partial\tilde g/\partial p(p,0)\neq 0$ at points of $\Sigma_{1,1}$.
Therefore, $d \tilde g$ and $d \tilde g_s$ are linearly independent at $\Sigma_{1,1}$.
It now follows (see Def. \ref{def cusp}) that $\pi_R$
has a cusp singularity at $\Sigma_{1,1}$, and the lemma is proved.
\end{proof}

\subsection{Criteria for absence of artifacts}\label{subsec absence}

In general, even if through beam forming the microlocal support of $\mathcal F$ is restricted to a set where $\mathcal C$ is a local canonical graph,
 if $\pi_L$ is not injective, i.e., if the Bolker condition \eqref{eqn Bolker} is violated,
one can expect that the backprojected image $\mathcal F^*W$ will contain artifacts.
Recalling (\ref{left_uv}), we see that the question of whether artifacts will be present in the backprojected 
image boils down to the whether or not, for a fixed value of $u$,  the map
\begin{eqnarray*}
\mathcal V: v\mapsto \ddot R
\end{eqnarray*}
is injective.
So, using \eqref{rddot}, suppose that we specifiy a value
\begin{eqnarray*}
\alpha := \mathcal V(v) = \rho \frac{\sqrt{1-u^2}}{h} \left( \frac{1+hS-u^2}{C} \right)
\end{eqnarray*}

Then we have that
\begin{eqnarray}
h\alpha(e^v+e^{-v})=\rho \sqrt{1-u^2} \left\{ 2(1-u^2)+h(e^v-e^{-v}) \right\} \nonumber \\
\Rightarrow h\beta_- e^v-h\beta_+ e^{-v} +2\rho (1-u^2)^{3/2}  = 0, \label{e1}
\end{eqnarray}
where
\begin{eqnarray*}
\beta_\pm :=\rho \sqrt{1-u^2}\pm\alpha\ .
\end{eqnarray*}
We cannot have both $\beta_+=0, \beta_-=0$, since $\rho>0, u^2\neq 1$. We consider both cases separately, as follows.

{\em Case 1:} Assume $\beta_-\neq 0$. Divide across (\ref{e1}) by $h\beta_-$ to get
\begin{eqnarray} \label{e2}
e^v-\left( \frac{\beta_+}{\beta_-} \right) e^{-v} + 2\eta  = 0,
\end{eqnarray}
where
\begin{eqnarray*}
\eta  := \frac{\rho(1-u^2)^{3/2}}{h\beta_-} .
\end{eqnarray*}
Let $y:=e^v$ and multiply equation (\ref{e2}) by $y$ to get the quadratic equation
\begin{eqnarray*}
y^2+  2\eta  y -\beta_+/\beta_- = 0.
\end{eqnarray*}
The solutions of this equation are
\begin{eqnarray}
y_\pm =-\eta \pm \sqrt{\eta ^2+\beta_+/\beta_-}  \nonumber \\
 \Rightarrow y_\pm =-\eta \left(1\pm\sqrt{1+\beta_+/(\eta ^2\beta_-)}\ \right) \label{roots}
\end{eqnarray}
From the definition of $y$, we must have $y_\pm =e^{v_\pm}>0$ for some value of $v_\pm$ and therefore,  $y_\pm>0$.

Next we  study the term under the square root in \eqref{roots}:
\begin{eqnarray*}
\frac{\beta_+}{\eta ^2\beta_-}=\frac{h^2}{\rho^2(1-u^2)^3}\left\{ \rho^2(1-u^2)-\alpha^2\right\} \ .
\end{eqnarray*}
Since $\eta $ has the same sign as $\beta_-$, the  map $\mathcal V$ will be injective if
\begin{eqnarray} \label{rho_alpha}
\alpha^2 < \rho^2(1-u^2)
\end{eqnarray}
since then, $y_+$ is the only possible positive root if $\beta_- <0$, while $y_-$ is the only possible root if $\beta_->0$.

{\em Case 2:} Assume $\beta_+\neq 0$. Dividing (\ref{e1}) across by $h\beta_+$ we obtain
\begin{eqnarray} \label{e2plus}
e^{-v}-\left( \frac{\beta_-}{\beta_+} \right) e^{v} - 2\tilde\eta  = 0,
\end{eqnarray}
where
\begin{eqnarray*}
\tilde{\eta }=  \frac{\rho(1-u^2)^{3/2}}{h\beta_+} .
\end{eqnarray*}
Letting $\tilde y:= e^{-v}$ and multiplying (\ref{e2plus}) across by $\tilde y$ we get
\begin{eqnarray*}
\tilde y^2-2\tilde{\eta }\tilde y-\frac{\beta_-}{\beta_+}=0.
\end{eqnarray*}
This quadratic equation has solutions
\begin{eqnarray*}
\tilde y_\pm = \tilde \eta \pm \sqrt{\tilde{\eta }^2+\beta_-/\beta_+}\\
\Rightarrow \tilde y_\pm = \tilde\eta \left(1\pm\sqrt{1+(\beta_-/(\tilde\eta ^2\beta_+)}\ \right) .
\end{eqnarray*}
One can easily check that condition (\ref{rho_alpha}) also guarantees that $\beta_-/({\tilde\eta ^2\beta_+)}>0$, which establishes injectivity of $\mathcal V$ in this case too.

\begin{lemma}\label{lem criterion}
  The following inequality guarantees injectivity of $\mathcal V$:
\begin{eqnarray} \label{inj_cond}
(\bm x-\rho \bm e) \cdot \bm e < \rho\left( \frac{h^2-1}{2} \right)\ .
\end{eqnarray}
\end{lemma}
 \begin{proof}
In fact, if (\ref{inj_cond}) holds, then from \eqref{ecomponent} we have
\begin{equation}
hS  < \frac{h^2-1}{2} \ .
\end{equation}
Since $0 \leq u^2 \leq 1$, this implies the inequality $1-u^2+2hS  < h^2$; the left of which is decreased by multiplying by $1-u^2$ to obtain
$(1-u^2)(1-u^2+2hS)  < h^2 = h^2 (C^2 - S^2) $.   Moving $h^2 S^2$ to the left side, we obtain
$ (1-u^2+hS)^2 < h^2C^2  $, which can be rewritten as
$ \frac{\rho^2(1-u^2)}{h^2}\left(\frac{ 1-u^2+hS}{C}\right)^2  < \rho^2(1-u^2)$.
This is exactly condition (\ref{rho_alpha}), as claimed,
so that (\ref{inj_cond}) is sufficient for injectivity of $\mathcal V$.
\end{proof}

Writing \eqref{inj_cond} as $\bm x \cdot \bm e < \rho + \rho (h^2 -1)/2$ and considering all possible locations on the flight path, it then follows that a sufficient condition to guarantee injectivity of  $\mathcal V$ is
\begin{eqnarray}
|\bm x|< \rho\left( \frac{h^2+1}{2} \right) \quad  \Leftrightarrow \quad \bm x\in \mathcal D_{h,\rho/2} \ . \label{inj}
\end{eqnarray}

\noindent{{This concludes the proof of Theorem \ref{thm circle}.\qed}}

\begin{remark}
Note that condition (\ref{inj}) implies that the  left projection  $\pi_L$
is injective, and also guarantees that condition \eqref{dop_cond1} is automatically satisfied,
so that $\pi_L$ is an immersion.
Therefore, $\mathcal C$ is a canonical graph over the region defined by \eqref{inj},  ensuring artifact-free imaging of scenes there via filtered backprojection.
\end{remark}

\begin{remark}\label{rem circular}
To summarize the microlocal analysis so far of DSAR for a circular flight path,   $\mathcal C$ is
a canonical relation whose projections, $\pi_L$ and $\pi_R$,  are of  fold and cusp type, resp.
A similar geometry
appeared in \cite{FeNo15} in the case of monostatic SAR when the flight path had simple inflection points.
In that situation it was shown that $\mathcal F^* \mathcal F$ produces  an artifact relation which is a canonical relation,
with a codimension-two set of  points where it is nonsmooth, called an \emph{open umbrella};
see \cite{FeGr10} for background material on umbrellas and their relevance in seismic imaging, and \cite{FeNo15} for how this geometry arises in monostatic SAR.
\end{remark}

\section{First order correction to the start-stop approximation}\label{sec beyond}

{ We now describe and analyze a correction to  the  start-stop approximation,  Assumption \ref{startstop}, which in the form of
$R(t+T_{tot})=R(t)$
was substituted into \eqref{eqn Ttot} and used to derive \eqref{STFT} and thus \eqref{Wfinal}.
The correction we now consider comes from modifying the discussion below \eqref{def R} by
including a first order term in the expansion of $R(t+T_{tot})$.
 See also \cite{Tsy09,CB2011} for discussions of the start-stop approximation and its limitations.

As in Sec. \ref{sec Doppler}, $T_{tot}$ is determined implicitly by \eqref{eqn Ttot}, namely $c_0T_{tot}=R(t)+R\left(t+T_{tot}\right)$.
However, we now expand
$$c_0T_{tot}=R(t)+R\left(t+T_{tot}\right)\approx 2R(t)+\dR(t)\cdot T_{tot}$$
Solving for $T_{tot}$ and ignoring terms $\mathcal O(c_0^{-3})$, we obtain the first order refinement of  the start-stop approximation, namely
that the total  travel time is
\begin{equation}\label{eqn improved}
T_{tot}\approx 2\coi R(t)+ 2\cot R(t)\dR(t)
\end{equation}
and thus, under this refined approximation, the scattered wave  arrives at time
$$t_{sc}= t+2\coi R(t)+2\cot R(t)\dR(t).$$
Using this, the analogue of $W_0$ from (\ref{STFT}) is}

\begin{align}	\label{STFTcorrected}
W_1(s, \omega) &:= \int e^{i \omega (t-s)} \ell(\omega_0(t-s)) d(t) dt\nonumber\\
&= \int e^{i \omega (t-s)} \ell(\omega_0(t-s)) \int \frac{e^{-i \omega_0 \left(t - 2\left[\coi R\left(t\right)
+\cot R\left(t\right)\dR\left(t\right)\right]\right)}} {\left(4 \pi R\left(t\right)\right)\left(4 \pi R\left(t_{sc}\right)\right)}  V(\bm x) d\bm x dt
\end{align}
Substituting into \eqref{STFTcorrected} the linear terms of the Taylor expansions
\begin{equation}	\label{Taylor2}
R(t) = R(s) + \dot{R}(s) (t-s) + \cdots,\quad \dR(t)=\dR(s)+\ddR(s)(t-s)+\cdots,
\end{equation}
and, calculating {{modulo}} $(t-s)^2$,
yields the approximation $W_1$ of $W$ analogous  to (\ref{Wfinal}):
\begin{equation}\label{Wfinalcorrected}
W_1(s, \omega) = \int e^{i\phi(s,\omega,\bm x;\tau)} a(\bm x, s;\tau) V(\bm x) d\tau d\bm x,
\end{equation}
where

\begin{equation}\label{def phicorrected}
\phi=\tau\left(\omega-\omega_0+2\omega_0\left[\coi\dR\left(s\right)+\cot \left(R\left(s\right)\ddR\left(s\right)+\dR(s)^2\right)\right]\right)
\end{equation}
and
\begin{equation}\label{def ampcorrected}
a(\bm x, s;\tau)=
e^{i \omega_0\left[ 2\left(\coi R\left(s\right)+\cot R\left(s\right)\dR\left(s\right)\right)-  s\right]}
(4\pi)^{-2}R\left(s\right)^{-1}R\left(s_{sc}\right)^{-1} \ell(\omega_0\tau).
\end{equation}
Rewriting the expression in \eqref{def phicorrected} that multiplies $\cot$ as $\left(1/2\right)\left(R^2\right)^{\bm{..}}$
for compactness (where the superscript upper right dots   still denote differentiation in $t$), one sees that
this modified phase function $\phi$ parametrizes the canonical relation

\beqa\label{eqn Cphicorrected}\nonumber
\mathcal C_{mod}\!\!\! &=&\!\!\! \Big\{ \Big( s, \omega_0 - 2 \omega_0\left[\coi \dot{R} +\left(1/2\right)\cot\left(R^2\right)^{\bm{..}}\right],
2  \tau\omega_0\left(\coi  \ddot{R} +\cot(1/2)(R^2)^{\bm{\cdots}}\right), \tau; \cr
& &\quad\, \bm x, -2 \tau \omega_0\left[\coi {d}_{\bm x} \dot{R}
+\left(1/2\right)\cot {d}_{\bm x}\left(R^2\right)^{\bm{..}} \right] \Big)
: s\in \mathbb{R}, \bm x \in \mathbb{R}^2, \tau \in \mathbb{R} \setminus 0  \Big\}.
\eeqa
\medskip

The fold/blowdown structure of the canonical relation in the case of the linear flight track, established in \S\ref{sec straight},
is \emph{not} structurally stable, since arbitrarily small perturbations of blowdowns are not (in general) blowdowns.
Nevertheless, we now show that under the first order correction to
the start-stop approximation, $\mathcal C_{mod}$ is still a fold/blowdown, indicating the robustness of our approach.
\medskip

To see this, note that the analogue of (\ref{eqn pir}) taking the correction into account is
\be\label{eqn pir corrected}
 \frac{D(\bxi)}{D(s, \tau)} =-2\omega_0\left[\tau\ddot\Gamma,\, \dot\Gamma\right],
 \ee
 where
\be\nonumber
\Gamma(\bm x,s)=\coi d_{\bm x} R+(1/2)\cot \dot{d}_{\bm x}(R^2)=\coi d_{\bm x}R+\cot\left[\dR d_{\bm x}R+R d_{\bm x}\dR\right].
 \ee
(Thus, $\mathcal C_{mod}$  is nondegenerate at a point  if and only if  $\Gamma$ is a smooth curve with nonzero curvature for the corresponding $\bm x, s$.) One computes
\be\label{eqn Gammadot}
\dot{\Gamma}=\coi \nx\dR+\cot\left[\ddR\nx R + {2} \dR\nx \dR+R\nx \ddR\right]
\ee
and
\be\label{eqn Gammadotdot}
\ddot{\Gamma}=\coi\nx\ddR +\cot\left[ \dddot R\nx R + {3}\ddR\nx\dR + {3}\dR\nx\ddR+R\nx\dddot R\right]
\ee
Using (\ref{straight 1}-\ref{straight 5}), one finds
$$\ddR \nx R + {2} \dR\nx \dR+R \nx \ddR = \ 0\hbox{ and }\dddot R\nx R + {3}\ddR\nx\dR + {3}\dR\nx\ddR+R\nx\dddot R=0.$$
Thus
\be\label{eqn Gdstraight}
\dot{\Gamma}=(\coi R^{-3})\left[-\left(x_2^2+h^2\right),\left(x_1-s\right)x_2\right].
\ee
Differentiating this directly yields

$$\ddot{\Gamma} = c_0^{-1} R^{-5}\big( -3(x_1 - s) (x_2^2 + h^2) , \  2 x_2 R^2 - 3 x_2 (x_2^2 + h^2)\big), $$
so that

$$\det\left[\ddot{\Gamma},\dot{\Gamma}\right]= \left(\frac{1}{c_0^2R^{6}}\right)x_2(x_2^2+h^2),$$
which vanishes to first order at $(s,\bm x,\tau)\in\Sigma=\{x_2=0\}$. Furthermore,  at these points,
ker$(D\pi_R)=\R\cdot\frac{\partial}{\partial s}\subset T\Sigma$, so that $\pi_R$ has a blowdown singularity at $\Sigma$.
A similar calculation shows that $D\pi_L$ has a kernel whose nonzero elements have a nonzero coefficient of
$\frac{\partial}{\partial x_2}$, so that $\pi_L$ has a fold singularity at $\Sigma$, and $\mathcal C_{mod}$ is a fold/blowdown
 canonical relation.
 \medskip

In the case of the circular flight track, the fold/cusp structure of the canonical relation
derived in \S\ref{sec circular} under the start-stop approximation,
is structurally stable in the following sense: { a small perturbation of the phase function
(say in the $C^4$ topology) will result in a small perturbation of $\mathcal C$
and this causes small $C^3$ perturbations of the projections $\pi_L$ and $\pi_R$.
Folds are stable under $C^2$ perturbations and cusps are stable under $C^3$ perturbations \cite{GoGu},
and thus one expects that the first order correction to the start-stop approximation will not essentially change
the microlocal analysis for a circular flight path:
 $\mathcal C_{mod}$ will still be a fold/cusp,
 so that the artifacts are of the same type and strength as shown above,
 although with their locations  moved slightly. However, we have not proven this and it is a subject for future research.}

\section{Concluding Remarks}\label{sec conclusion}

We briefly compare our results for Doppler SAR with the case of monostatic SAR treated in \cite{NC04}.
In all cases, we can expect the strength of any associated artifacts in the image to be as strong as the
bona-fide part.

For the case of a linear flight path, the results for Doppler are the same as for monostatic SAR,
in the sense that $\pi_L$ and $\pi_R$ have fold and blowdown singularities, resp.
The normal operator $\mathcal F^*\mathcal F$ formed without beam forming will have the same strong artifact,
i.e., $\mathcal F^*\mathcal F\in I^{-1,0}(\Delta, C_{\chi})$, with $C_\chi$ a canonical graph
(cf. Remark \ref{rem linear} in Sec. \ref{sec straight}).

\par On the other hand, for  a circular flight path, the microlocal geometry for  Doppler SAR differs from that for
monostatic SAR:  we have shown that in the case of  Doppler, the singularities of $\pi_L$, $\pi_R$,
are of fold and  cusp type, resp.,
whereas  for monstatic SAR, both singularities are  folds \cite{NC04}.
Regarding the operator $\mathcal F^* \mathcal F$, 
it was shown in \cite{Fe05} that in the monostatic SAR, $\mathcal F^* \mathcal F \in I^{-1,0}(\Delta, \tilde{C})$,
where $\tilde{C}$ is a two-sided fold. For Doppler SAR we expect, based on the results of Felea and Nolan \cite{FeNo15},
that $\mathcal F^* \mathcal F  \in I^{-1}(\Delta, \tilde{C})$, where $\tilde{C}$ is an open umbrella (see Remark \ref{rem circular}).
For the circular flight  path, we  have described an explicit region $D_{h,\rho}$ where the wave front relation of $F$ is a canonical graph, so that no artifacts appear at the back projection, allowing accurate imaging of the terrain.
In addition, as described at the end of Sec. \ref{subsec prelim},
the artifacts arising from Doppler imaging in a circular flight path
geometry  are spatially separated from their location for monostatic SAR for the same flight path.
It should therefore be possible to identify and remove artifacts by combining
monostatic SAR data with Doppler SAR data; we hope
to return to this in the future.

\section*{Acknowledgments}
 This paper originated in joint work  with Margaret Cheney, and we  thank her for suggesting the problem and for her contributions.
 The authors were supported by an AIM  Structured
Quartet Research Ensemble (SQuaRE) at the American Institute of Mathematics, San Jose, CA.
They also participated in the Fall, 2017, program on Mathematical and Computational Challenges
in Radar and Seismic Reconstruction at ICERM, Providence, RI. AG was  supported by NSF award DMS-1362271.

\end{document}